\documentclass[reqno,centertags, 12pt]{amsart}
\usepackage{amsmath,amsthm,amscd,amssymb}
\usepackage{esint}  %from Maxim for circular integral
\usepackage{latexsym}
\usepackage{graphicx}
\usepackage{enumitem}
\usepackage[mathscr]{eucal}
\newcommand{\bbC}{{\mathbb{C}}}
\newcommand{\bbD}{{\mathbb{D}}}

\newcommand{\bbR}{{\mathbb{R}}}

\newcommand{\fre}{{\mathfrak{e}}}
\newcommand{\frf}{{\frak{f}}}

\newcommand{\al}{\alpha}
\newcommand{\be}{\beta}
\newcommand{\de}{\delta}

\newcommand{\calC}{{\mathcal{C}}}

\newcommand{\calH}{{\mathcal H}}

\newcommand{\calW}{{\mathcal W}}

\newcommand{\lb}{\label}
\newcommand{\f}{\frac}

\newcommand{\ol}{\overline}

\newcommand{\inte}{\text{\rm{int}}}

\newcommand{\bi}{\bibitem}
\newcommand{\hatt}{\widehat}
\newcommand{\beq}{\begin{equation}}
\newcommand{\eeq}{\end{equation}}
\newcommand{\ba}{\begin{align}}
\newcommand{\ea}{\end{align}}
\newcommand{\veps}{\varepsilon}

\newcommand{\st}{\,|\,}

\renewcommand{\Im}{\operatorname{Im}}

\newcounter{smalllist}

\newcommand{\comm}[1]{}

\allowdisplaybreaks
\numberwithin{equation}{section}
\newtheorem{theorem}{Theorem}[section]
\newtheorem{proposition}[theorem]{Proposition}
\newtheorem{lemma}[theorem]{Lemma}
\newtheorem{corollary}[theorem]{Corollary}
\theoremstyle{definition}
\newtheorem{example}[theorem]{Example}
\newtheorem{conjecture}[theorem]{Conjecture}

\newtheorem*{remarks}{Remarks}

\newcommand{\norm}[1]{\lVert#1\rVert}

\begin{document}

\title[Chebyshev Polynomials, IV]{Asymptotics of Chebyshev Polynomials,\\ IV. Comments on the Complex Case}
\author[J.~S.~Christiansen, B.~Simon and M.~Zinchenko]{Jacob S.~Christiansen$^{1,4}$, Barry Simon$^{2,5}$ \\and
Maxim~Zinchenko$^{3,6}$}

\thanks{$^1$ Centre for Mathematical Sciences, Lund University, Box 118, 22100 Lund, Sweden.
 E-mail: stordal@maths.lth.se}

\thanks{$^2$ Departments of Mathematics and Physics, Mathematics 253-37, California Institute of Technology, Pasadena, CA 91125.
E-mail: bsimon@caltech.edu}

\thanks{$^3$ Department of Mathematics and Statistics, University of New Mexico,
Albuquerque, NM 87131, USA; E-mail: maxim@math.unm.edu}

\thanks{$^4$ Research supported in part by Project Grant DFF-4181-00502 from the Danish Council for Independent Research.}

\thanks{$^5$ Research supported in part by NSF grants DMS-1265592 and DMS-1665526 and in part by Israeli BSF Grant No. 2014337.}

\thanks{$^6$ Research supported in part by Simons Foundation grant CGM-581256.}

\

\date{\today}
\keywords{Chebyshev polynomials, Lemniscates, Zero Counting Measures, Totik--Widom upper bound}
\subjclass[2010]{41A50, 30C80, 30C10}

\begin{abstract}  We make a number of comments on Chebyshev polynomials for general compact subsets of the complex plane.  We focus on two aspects: asymptotics of the zeros and explicit Totik--Widom upper bounds on their norms.
\end{abstract}

\maketitle

%%%%%%%%%%%%%%%%%%%%%%%%%%%%%%%%%%%%%%%%%%%%%%%%%%%%%%%%%%%%%%
\section{Introduction} \lb{s1}
%%%%%%%%%%%%%%%%%%%%%%%%%%%%%%%%%%%%%%%%%%%%%%%%%%%%%%%%%%%%%%

Let $\fre \subset \bbC$ be a compact, not finite, set.  For any continuous, complex-valued function, $f$, on $\fre$, let
\begin{equation}\label{1.1}
  \norm{f}_\fre = \sup_{z \in \fre} |f(z)|
\end{equation}
The Chebyshev polynomial, $T_n$, of $\fre$ is the (it turns out unique) degree $n$ monic polynomial that minimizes $\norm{P}_\fre$ over all degree $n$ monic polynomials, $P$.  We define
\begin{equation}\label{1.2}
  t_n \equiv \norm{T_n}_\fre
\end{equation}
We will use $T_n^{(\fre)}$ and $t_n^{(\fre)}$ when we want to be explicit about the underlying set.  We let $C(\fre)$ denote the logarithmic capacity of $\fre$ (see \cite[Section 3.6]{HA} or \cite{AG, Helms, Landkof, MF, Ransford} for the basics of potential theory; in particular, we will make reference below to the notion of equilibrium measure).  

This paper continues our study \cite{CSZ1, CSYZ2,CSZ3} of $t_n$ and $T_n$.  Those papers mainly (albeit not entirely) dealt with the case $\fre\subset\bbR$.  In this paper, we make a number of comments on the general complex case focusing on two aspects, upper bounds on $t_n$, which we called Totik--Widom bounds (henceforth, sometimes, TW bounds), and the asymptotics of zeros of $T_n(z)$.  As is often the case in complex analysis, there is magic in simple observations.  Larry Zalcman has long been a master magician in this way, so we are pleased to provide this present to him recognizing his long service as editor-in-chief of Journal d'Analyse Math\'{e}matique.

We begin by sketching the uniqueness proof for $T_n$ which extends the argument when $\fre\subset\bbR$ (a case that appears in many places including \cite{CSZ1}).  We call $z\in\fre$ an extreme point for $P$ if and only if $|P(z)|=\norm{P}_\fre$.  We claim that any norm minimizer, $P$, a monic polynomial of degree $n$, must have at least $n+1$ extreme points.  For, if there are only $z_1,\dots,z_k$ with $k\le n$ distinct extreme points for $P$, by Lagrange interpolation, we can find a polynomial $Q$ with degree $k-1$ so that
\begin{equation}\label{1.3}
  Q(z_j)=P(z_j), \quad j=1,\dots,k
\end{equation}
Then for $\varepsilon$ small and positive, it is easy to see that $\norm{P-\varepsilon Q}_\fre < \norm{P}_\fre$ violating the fact that $P$ is a norm minimizer.  (We note that for $\fre = [-1,1]$, $T_n$ has exactly $n+1$ extreme points although for many sets, e.g.\ $\fre=\bbD$, each $T_n$ has infinitely many extreme points.)

Suppose now that $f$ and $g$ are both norm minimizers among monic polynomials of degree $n$.  Then so is $h=\frac{1}{2}(f+g)$.  Pick $\{z_j\}_{j=1}^{n+1}$ distinct extreme points for $h$.  Since $|h(z_j)|=t_n$ and $|f(z_j)|, |g(z_j)|\le t_n$, we must have that $f(z_j)=g(z_j)$ for $j=1,\dots,n+1$.  Since $\deg(f-g)\le n-1$, we have that $f=g$ completing the proof of uniqueness of the minimizing polynomial.

Recall (see, e.g.\ \cite{HA}) that the outer boundary, $O\partial K$, of a compact set, $K$, is the boundary of the unbounded component of $\bbC\setminus K$.  A compact set is simply connected if and only if its boundary equals its outer boundary.  Given a compact set $K$, there is a unique compact set, $\hatt{K}$, with $\partial\hatt{K}=O\partial K$.  By the maximum principle, $\norm{f}_K=\norm{f}_{\hatt{K}}$ for any entire analytic function $f$, so $K$ and $\hatt{K}$ have the same Chebyshev polynomials.  They also have the same potential theory (i.e.\ if $C(\cdot)$ is the capacity, then $C(\hatt{K})=C(K)$; they have the same potential theory Green's function in the region where that function is positive and the same equilibrium measure).  Thus, without loss of generality, we will often state results only for simply connected sets.  Some authors use ``simply connected'' for ``connected and simply connected'' so sometimes, we will say ``not necessarily connected, simply connected'' to emphasize that we do not.

Let $w_1,\dots,w_n$ be the zeros of $T_n$ counting multiplicity and $\mu_n=\frac{1}{n}\sum_{j=1}^{n}\delta_{w_j}$ the normalized counting measure for zeros of $T_n$.  The limit points of $\{\mu_n\}_{n=1}^\infty$ as $n\to\infty$ are called density of Chebyshev zeros for $\fre$.

In \cite{ST1990}, Saff--Totik proved the following theorem:

\renewcommand\thetheorem{\Alph{theorem}}
\begin{theorem} [\cite{ST1990}] \lb{TA} Let $K$ be a connected and simply connected set so that $K^{int}$ is also connected.  Then

(a) If $K$ is an analytic Jordan region (i.e.\ $\partial K$ is an analytic simple curve), then there is a neighborhood, $N$, of $\partial K$ so that for all large $n$, $T_n^{(K)}$ has no zeros in $N$.

(b) If $\partial K$ has a neighborhood, N, and there is a sequence $n_j \to \infty$ so that $\mu_{n_j}(N) \to 0$, then $K$ is an analytic Jordan region.
\end{theorem}

In \cite{BSS1988}, Blatt--Saff--Simkani proved the following theorem

\begin{theorem} [\cite{BSS1988}] \lb{TB}  Let $K$ be a compact, simply connected subset of $\bbC$ whose interior is empty and with $C(K)>0$.  Then, as $n \to \infty$, the Chebyshev zero counting measure, $\mu_n$, converges to $\mu_K$, the equilibrium measure for $K$.
\end{theorem}

In Section \ref{s2}, we will explore local versions of these theorems and prove

\renewcommand\thetheorem{\arabic{theorem}}
\numberwithin{theorem}{section}
\setcounter{theorem}{0}

\begin{theorem} \lb{T1.1} Let $\fre$ be a simply connected (but not necessarily connected), compact subset of $\bbC$, regular for potential theory, and $N$ an open, connected, simply connected set so that $N\cap\partial\fre$ is a continuous arc that divides $N$ into two pieces, $\fre^{int}\cap N$ and $(\bbC\setminus\fre)\cap N$.  Let $M_n(N)$ be the number of zeros of $T_n^{(\fre)}$ in $N$.  Then either $\liminf M_n(N)/n > 0$ or $N\cap\partial\fre$ is an analytic arc.
\end{theorem}

In particular, if $K$ is a Jordan region whose boundary curve is nowhere analytic, all of the boundary points are points of density of zeros of $T_n$.  We believe, but cannot prove, that in this case the density of zeros measure exists and is the equilibrium measure.  Moreover, this theorem implies that if $\partial\fre$ is piecewise analytic but not analytic at some corner points, then at least these corner points are points of density for the zeros.  In Section 2, we will also discuss zeros near crossing points of a boundary as occur for example with a figure eight.

\begin{theorem} \lb{T1.2} Let $\fre$ be a compact, simply connected subset of $\bbC$ and $N$ an open, connected, simply connected set so that $C(N\cap\fre)>0$ and so that $N\cap\fre$ has two-dimensional Lebesgue measure zero.  Then as $n \to\infty$, the density of zeros measure for $T_n^{(\fre)}$ restricted to $N$ converges weakly to the equilibrium measure, $\mu_\fre$, restricted to $N$.
\end{theorem}

The example to think of is $\fre=\overline{\bbD}\cup [1,2]$ and $N$ a small disk about some point $x_0\in (1,2)$.  We suspect that the measure zero condition can be replaced by the condition that the set has empty interior; we will discuss this further in Section \ref{s2}.

The last three sections deal with TW upper bounds on $t_n^{(\fre)}$.  One knows that $t_n^{1/n}\to C(\fre)$ (see, for example the discussion in \cite[Theorem 4.3.10]{OT}) and that $C(\fre)^n\le t_n$ (the Szeg\H{o} lower bound).  In \cite{CSZ1,CSYZ2}, a key in the analysis of the pointwise asymptotics of $T_n(z)$ were upper bounds on $t_n$ of the form
\begin{equation}\label{1.4}
  t_n \le QC(\fre)^n
\end{equation}
which we dubbed Totik--Widom bounds after Widom \cite{Widom} and Totik \cite{Totik09} who proved it for finite gap subsets of $\bbR$ (and Widom also for finite unions of disjoint $C^{2+\varepsilon}$ Jordan regions).

Our work in \cite{CSZ1,CSYZ2} relied on Parreau--Widom sets (after \cite{Parreau, Widom71}).  For any compact set $\fre\subset\bbC$, let $\calC$ be the critical points of $G_\fre$, the Green's function for $\fre$, on the unbounded component of $\bbC\setminus\fre$.  Define
\begin{equation} \lb{1.5}
PW(\fre) \equiv \sum_{w \in \calC} G_\fre(w)
\end{equation}
The Parreau--Widom sets are those with $PW(\fre)<\infty$.  In \cite{CSZ1}, we proved for sets regular for potential theory that 
\begin{equation}\label{1.6}
  \fre\subset\bbR \; \Rightarrow \; t_n^{(\fre)} \le 2 \exp\bigl\{PW(\fre)\bigr\}C(\fre)^n
\end{equation}
and in \cite{CSYZ2}, we proved that if $\fre\subset\bbR$ has the property that for all decomposition $\fre = \fre_1 \cup \dots \cup \fre_\ell$ into closed disjoint sets, one has that $\mu_\fre(\fre_1),\dots,\mu_\fre(\fre_\ell)$ are rationally independent, then \eqref{1.4} for $\fre$ implies that $PW(\fre)<\infty$.  In \cite{CSZ3}, we proved that for such sets with $PW(\fre)<\infty$, one has that the set of limit points of $\calW_n(\fre) \equiv \norm{T_n}_\fre/C(\fre)^n$ (called Widom factors) is the entire closed interval $[2,2\exp\{PW(\fre)\}]$.  Sections \ref{s3}--\ref{s5} explore the question of when a bound like \eqref{1.4} holds for some $\fre\subset\bbC$.  In \cite{CSZ1}, we raised the question of whether $\calW_n(\fre)$ is bounded for every PW set in $\bbC$ and we will discuss this further in section \ref{s5}.

Sections \ref{s3} and \ref{s4} discuss two cases where we can prove TW bounds with explicit constants (for many but not all of these sets, Widom has TW bounds but without explicit constants.  Basically, the sets which are not handled in \cite{Widom} include certain unions of mutually external analytic Jordan curves but some of which can touch at single points).

Section \ref{s3} discusses solid lemniscates, that is, sets of the form
\begin{equation}\label{1.7}
  \frf_n = \{z\in\bbC \st |P_n(z)| \le \alpha\}
\end{equation}
for a polynomial $P_n$ of exact degree $n$.  In \cite{CSZ3} (see also Faber \cite{Faber}), we implicitly noted that if $P_n$ is monic, then
\begin{equation}\label{1.8}
  T_{kn}^{(\frf_n)}(z) = P_n(z)^k; \qquad k=1,2,\dots
\end{equation}
and we use this in Section \ref{s3} to prove that

\begin{theorem} \label{T1.3} Let $\frf_n$ be of the form \eqref{1.7}.  Define
\begin{equation}\label{1.9}
  Q = \max_{j=0,\dots,n-1} \calW_j(\frf_n)
\end{equation}
Then, \eqref{1.4} holds for $\frf_n$.
\end{theorem}

Our discussion in Section \ref{s4} is motivated by an old result of Faber \cite{Faber} (he stated it for $[-1,1]$; we use $[-2,2]$ to minimize factors of $2$.  The results are equivalent).

\setcounter{theorem}{2}
\renewcommand\thetheorem{\Alph{theorem}}

\begin{theorem} [\cite{Faber}] \label{TC} Let $\fre$ be an ellipse with foci at $\pm 2$.  Then $T_n^{([-2,2])}$ (which are scaled multiples of the classical Chebyshev polynomials of the first kind) are the Chebyshev polynomials for $\fre$.
\end{theorem}

We note that $[-2,2]$ is the image of $\partial\bbD$ under the Joukowski map $z\mapsto x=z+z^{-1} \, (z=e^{i\theta}\Rightarrow x=2\cos\theta)$.  Let $z(x) = \frac{1}{2}[x+\sqrt{x^2-4}]$ where we take the branch of square root on $\bbC\setminus [-2,2]$ that behaves like $x$ near $x=\infty$.  Then the Green's function for $[-2,2]$ is $\log z(x)$ so $C([-2,2])=1$.  The Chebyshev polynomials of $[-2,2]$ are given by
\begin{equation}\label{1.10}
  T_n^{([-2,2])}(x) = z(x)^n+z(x)^{-n}
\end{equation}
so for $\fre=[-2,2]$ we have $t_n^{(\fre)}=2=2C(\fre)^n$ and hence $\calW(\fre)=2$ (saturating a lower bound of Schiefermayr \cite{SchLB} for $\fre\subset\bbR$).

The ellipses with foci $\pm 2$ are precisely the sets of the form
\begin{equation}\label{1.11}
  \fre^\alpha = \{x\in\bbC \st |z(x)|=e^\alpha\}
\end{equation}
(Here and in the rest of the paper, the reader needs to be careful to distinguish $\fre^\alpha$ from $e^\alpha$!) for some $\alpha>0$.  By Theorem \ref{TC} and \eqref{1.10}, one has that
\begin{equation}\label{1.12}
  t_n^{(\fre^\alpha)} = e^{n\alpha}+e^{-n\alpha}
\end{equation}
The Green's function for $\fre^\alpha$ is $\log\,z(x)-\alpha$ so $C(\fre^\alpha)=e^\alpha$ and we have that
\begin{equation}\label{1.13}
  t_n^{(\fre^\alpha)} = (1+e^{-2n\alpha})C(\fre^\alpha)^n
\end{equation}
and hence ($\fre=[-2,2]$)
\begin{equation}\label{1.14}
  \calW(\fre^\alpha) = \frac{1}{2}(1+e^{-2n\alpha})\calW(\fre)
\end{equation}

Section \ref{s4} generalizes Faber's results.  Recall that a period-$n$ set is a subset $\fre_n\subset\bbR$ so that there is a degree $n$ polynomial $P_n$ with $\fre_n=P_n^{-1}([-2,2])\equiv\{z\in\bbC \st P_n(z)\in[-2,2]\}$. %(where $P_n$ is restricted to be one of those polynomials with $z\in\bbC$ and $P_n(z)\in[-2,2]\Rightarrow z\in\bbR$).
These are the spectra of period $n$ Jacobi matrices (see Geronimo--Van Assche \cite{GvA}, Peherstorfer \cite{Per0,Per1,Per2,Per3}, Totik \cite{Totik00,Totik01,Totik09,Totik12} or Simon \cite[Chap. 5]{SimonSzego}).  We will prove in Section \ref{s4} that

\renewcommand\thetheorem{\arabic{theorem}}
\numberwithin{theorem}{section}
\setcounter{theorem}{3}
\begin{theorem} \lb{T1.4} Let $\fre_n$ be a period-$n$ set and $G_n(z)$ its Green's function.  Let $\fre_n^\alpha=\{z \st G_n(z)=\alpha\}$ for some $\alpha>0$.  Then for $k=1,2,\dots$, one has that
\begin{equation}\label{1.15}
  T_{kn}^{(\fre_n^\alpha)} = T_{kn}^{(\fre_n)}
\end{equation}
and
\begin{equation}\label{1.15A}
  t_{kn}^{(\fre_n^\al)} = \cosh(kn\al) t_{kn}^{(\fre_n)}
\end{equation}
\end{theorem}

\begin{remarks} 1. If $n=1$, $\fre_n$ is just a single interval $\fre_1=[a,b]$ and this is just Faber's Theorem \ref{TC}.

2.  The result of \cite{CSZ3}, discussed further in Section \ref{s3}, that if $\frf_n$ is a lemniscate of the form \eqref{1.7}, then its Chebyshev polynomials of degree $nk$ are of the form \eqref{1.8} implies a complex version of Theorem \ref{T1.4}.  For the level sets of the Green's function of $\frf_n$ are again lemniscates with just a different value of $\alpha$ so the Chebyshev polynomials are the same since \eqref{1.8} holds for all values of $\al$.
\end{remarks}

\begin{theorem} \lb{T1.5} Let $\fre$ be a compact (not necessarily connected) simply connected subset of $\bbC$ which is regular for potential theory.  Let $G_\fre$ be its Green's function and $\fre^\alpha = \{z \st G_\fre(z)=\alpha\}$ for some $\alpha > 0$.  Then
\begin{equation}\label{1.16}
  \alpha\mapsto\calW_n(\fre^\alpha)
\end{equation}
is monotone decreasing in $\alpha$. In particular, if $\fre$ obeys a TW bound of the form \eqref{1.4}, so does each $\fre^\alpha$ with the same or smaller $Q$.
\end{theorem}

Given our result, \eqref{1.6}, of \cite{CSZ1}, we see that when $\fre\subset\bbR$ is a PW set, we have that
\begin{equation}\label{1.17}
  t_n^{(\fre^\alpha)} \le 2 \exp\{PW(\fre)\} C(\fre^\alpha)^n
\end{equation}
We will be able to improve this to

\begin{theorem} \lb{T1.6} If $\fre\subset\bbR$ is a PW set and $\fre^\alpha = \{z \st G_\fre(z)=\alpha\}$ for some $\alpha > 0$, then
\begin{equation}\label{1.18}
  t_n^{(\fre^\alpha)} \le (1+e^{-n\alpha})\exp\{PW(\fre)\} C(\fre^\alpha)^n
\end{equation}
\end{theorem}

\begin{remarks} 1. As $n\to\infty$, this beats \eqref{1.17} by a factor of $2$.

2.  We note that \eqref{1.18} has $PW(\fre)$ and not $PW(\fre^\alpha)$.  If $\fre$ is a finite gap set, and $\alpha$ is small, $PW(\fre^\alpha)=PW(\fre)-k\alpha$, where $k$ is the number of gaps.  As $\alpha$ increases, $PW(\fre^\alpha)$ shrinks further as $\fre^\alpha$ absorbs some of the critical points of $G_\fre$.
\end{remarks}

JC and MZ would like to thank F. Harrison and E. Mantovan for the hospitality of Caltech where some of this work was done.

%%%%%%%%%%%%%%%%%%%%%%%%%%%%%%%%%%%%%%%%%%%%%%%%%%%%%%%%%%%%%%
\section{Zero Counting Measure} \lb{s2}
%%%%%%%%%%%%%%%%%%%%%%%%%%%%%%%%%%%%%%%%%%%%%%%%%%%%%%%%%%%%%%

In this section, we study the asymptotics of the zero counting measure for Chebyshev polynomials and, in particular, prove Theorems \ref{T1.1} and \ref{T1.2}.  The theme is that in many ways the density of zeros wants to converge to the potential theoretic equilibrium measure for $\fre$. The only exception is when there are analytic pieces of $\partial\fre$.  We suspect this is true in much greater generality than we can prove it here (see the conjecture below).

The key to understanding this theme is

\setcounter{theorem}{3}
\renewcommand\thetheorem{\Alph{theorem}}

\begin{theorem} [\cite{SaT}] \label{TD} Outside the convex hull of $\fre$, one has that
\begin{equation}\label{2.1}
  |T_n(z)|^{1/n} \to C(\fre) \exp\{G_\fre(z)\}
\end{equation}
\end{theorem}

We provided another proof of this result as Theorem 3.2 of \cite{CSZ1}.  That proof was short.  The $\log$ of the ratio of the right to left side of \eqref{2.1} is a non-negative harmonic function on $\bbC\cup\{\infty\}\setminus\textrm{cvh}(\fre)$ by a theorem of Fej\'{e}r (which states that the zeros of $T_n$ lie within the convex hull of $\fre$) and by the Bernstein--Walsh lemma.  By the Faber--Fekete--Szeg\H{o} theorem (\cite{SzLB} or \cite[Theorem 4.3.10]{OT}), this harmonic function goes to zero at $\infty$ and so, by Harnack's inequality, everywhere on $\bbC\cup\{\infty\}\setminus\textrm{cvh}(\fre)$.

We next note the following theorem of Widom

\begin{theorem} [\cite{Widom67}] \lb{TE} Let $S$ be a closed subset of the unbounded component of $\bbC\setminus\fre$.  Then there is $N_S<\infty$ so that for all $n$, the number of zeros of $T_n$ in $S$ is at most $N_S$.
\end{theorem}

This implies
\renewcommand\thetheorem{\arabic{theorem}}
\numberwithin{theorem}{section}
\setcounter{theorem}{0}
\begin{theorem} \lb{T2.1} Any limit point, $d\mu_\infty$, of $d\mu_n$, the zero counting measure of $T_n$, is supported in $\fre$.  Moreover, for all $z$ in the unbounded component of $\bbC\setminus\fre$,
\begin{equation}\label{2.2}
  \int \log|z-w|\, d\mu_\infty(w) = \int \log|z-w|\, d\mu_\fre(w)
\end{equation}
where $d\mu_\fre$ is the equilibrium measure for $\fre$.
\end{theorem}

\begin{proof}  The first sentence is an immediate consequence of Theorem \ref{TE}.  Let $h$ be the difference of the two sides of \eqref{2.2} on the unbounded component of $\bbC\setminus\fre$.  By the first sentence, $h$ is harmonic.  By \eqref{2.1}, $h=0$ near infinity, so $h=0$ on all of the unbounded component of $\bbC\setminus\fre$ by the identity principle for harmonic functions (\cite[Theorem 3.1.17]{HA}).
\end{proof}

This theorem says that $d\mu_\fre$ is the balayage of $d\mu_\infty$ onto $\partial\fre$, equivalently, the balayage of $d\mu_n$ converges to $d\mu_\fre$; ideas that go back at least to Mhaskar--Saff \cite{MhS}.

The key to the proof of Theorem \ref{T1.1} is

\begin{proposition} \lb{P2.2}  Let $u$ be harmonic and not identically $0$ in a disk, $N$, centered at $z_0$ with $u(z_0)=0$.  Then, by shrinking the radius of $N$, if necessary, one can find $p\in\{1,2,\dots\}$ and $p$ analytic curves, $\gamma_1,\dots,\gamma_p$, with $\gamma_j(0)=z_0$ so that the angle between any two successive tangents, $\gamma_1'(0),\dots,\gamma_p'(0),-\gamma_1'(0),\dots,-\gamma_p'(0)$ is $\pi/p$ and so that
\begin{equation}\label{2.2a}
  \{z\in N \st u(z)=0\} = \{z\in N \st z \textrm{ lies in some } \gamma_j\}
\end{equation}
Moreover, the sign of $u$ alternates between successive sectors.
\end{proposition}

\begin{proof} There is a function $f$ analytic in $N$ so that $f(z_0)=0$ and $u(z) = \Im f(z)$.  By a standard result in complex analysis (see, for example \cite[Theorem 3.5]{BCA} and its proof), by shrinking $N$, if necessary, one can find $p\in\{1,2,\dots\}$ and an analytic function, $g$, in $N$ with $f=g^p$, $g(z_0)=0$, $g'(z_0)\ne 0$.  By another standard result in complex analysis (\cite[Theorem 3.4.1]{BCA}), $g$ has an analytic inverse function, $h$ (perhaps by shrinking $N$ further).  Let $\gamma_j(t)=h(te^{2\pi i(j-1)/p})$, $j=1,\ldots,p$ so $g(\gamma_j(t))=te^{\pi i(j-1)/p}$ and $f(\gamma_j(t))=t^p$ and $u(\gamma_j(t))=0$.  The remaining claims are immediate.
\end{proof}

\begin{proof} [Proof of Theorem \ref{T1.1}] If $\liminf M_n(N)/n = 0$, by passing to a subsequence and using compactness of the probability measures on the convex hull of $\fre$, we get a limit point, $d\mu_\infty$, of the zero counting measure with $\mu_\infty(N)=0$.  Let $z_0\in N\cap\partial\fre$.  It follows that $\int \log|z-w| d\mu_\infty(w)$ is harmonic near $z_0$.  By \eqref{2.2} and \cite[(3.6.43)]{HA},
\begin{equation}\label{2.4}
  u(z)=-\log(C(\fre))+\int \log |z-w|\,d\mu_\infty(w)
\end{equation}
equals $G_\fre$ outside $\fre$ and, in particular, is $0$ on $\partial\fre$ since $\fre$ is simply connected and regular for potential theory.

By \eqref{2.4}, $u(z)$ is subharmonic on $\bbC$.  By the maximum principle for subharmonic functions (\cite[Theorem 3.2.10]{HA}), $u(z) < 0$ on $\fre^{\inte}$.  Thus $\partial\fre=\{z \st u(z)=0\}$.

Since $z_0\in\partial\fre\cap N$, $u(z_0)=0$ and we can apply Proposition \ref{P2.2}.  We must have $p=1$ since otherwise, $\partial\fre$ doesn't divide $N$ into two pieces. Proposition \ref{P2.2} completes the proof.
\end{proof}

This argument is modelled after arguments in \cite{ST1990}.  They don't need an apriori assumption on $\partial\fre$ dividing $N$ in two since they make a global assumption on the zeros and, more importantly, they suppose that $\fre^{int}$ is connected.  If we don't make the apriori assumption on $\partial\fre$, we still have, by the above argument that

\begin{theorem} \lb{T2.3} Let $\fre$ be a simply connected, compact subset of $\bbC$ which is regular for potential theory.  Suppose that $z_0\in\partial\fre$ has a neighborhood, $N$, so that $\liminf M_n(N)/n = 0$.  Then for some  $p\in\{1,2,\dots\}$ there are $p$ analytic curves, $\gamma_1,\dots,\gamma_p$ through $z_0$ obeying the $\pi/p$ tangent condition so that (shrinking $N$, if necessary) $\fre\cap N$ is precisely the union of $p$ alternate sectors.
\end{theorem}

\begin{example} [Lemniscate of Bernoulli] \lb{E2.4} Consider the set
\begin{equation}\label{2.5}
  \fre = \{z \st |z^2-1| \le 1\}
\end{equation}
the simply connected, compact set bounded by the famous lemniscate of Bernoulli \cite{Bern}, a figure eight curve with crossing angle $\pi/2$.  By general principles (see \eqref{3.4} below), for $j=1,2,\dots$
\begin{equation}\label{2.6}
  T_{2j}(z)=(z^2-1)^j
\end{equation}
whose zeros are only at $z=\pm 1$, so the limit of the zero counting measure through the sequence of even orders is
\begin{equation}\label{2.6a}
  d\mu_\infty = \frac{1}{2}[\delta(z-1)+\delta(z+1)]
\end{equation}
which gives zero weight to the entire boundary of $\fre$.  We precisely have a point as in the last theorem with $p=2$.  Note that $T_n(-z)=(-1)^nT_n(z)$ by the uniqueness of Chebyshev polynomials so $T_{2j+1}(0)=0$.  We suspect (but cannot prove) that for $j$ large all the other zeros of $T_{2j+1}$ lie in small neighborhoods of $\pm 1$ and that the above $d\mu_\infty$ is also the limit through odd $n$'s.  The paper of Saff--Totik \cite{ST1990} shows that when $\fre^{\inte}$ is connected, one has that zero density on $\partial\fre$ implies no zeros at all in a neighborhood of $\partial\fre$.  If our surmise is correct, this example shows that that result does not extend when $\fre^{\inte}$ is not connected.  \qed
\end{example}

One Corollary of Theorem \ref{T1.1} is

\begin{corollary} \lb{C2.5} If $\fre$ is a Jordan curve whose boundary is nowhere analytic, then every point on the boundary is a limit of zeros of $T_n$
\end{corollary}

We suspect that much more is true.

\begin{conjecture} \lb{C2.6} \emph{If $\fre$ is a Jordan curve whose boundary is nowhere analytic, then the density of zeros measures converges to the equilibrium measure for $\fre$.}
\end{conjecture}

It is an intriguing question to understand when the density of zeros measure converges to the equilibrium measure.  An interesting result on this question is in Saff--Stylianopoulos \cite{SS2015} who prove that if $\partial\fre$ has an inward pointing corner in a suitable sense, then the density of zeros converges to the equilibrium measure.  For example, if $\fre$ is a polygon that is not convex, then their hypothesis holds.

It would be useful to know what happens for convex polygons; the simplest example is the equilateral triangle.  Theorem \ref{T1.1} implies that at least the vertices of the triangle are density points of zeroes.  We wonder what other points are density points of zeros (there must be others since the balayage of the average of the point masses at the corners is not the equilibrium measure).  It seems to us there are only two reasonable guesses.  Either the entire boundary are limit points of zeros (in which case it is likely the density of zeros converges to the equilibrium measure) or else the limit points are the skeleton obtained from the line segments from the centroid of the triangle to the vertices.  \cite[Figure 3]{SS2015},  which admittedly is for the Bergmann polynomials, not the Chebyshev polynomials, suggest the skeleton is the more likely answer.  We hope some numerical analyst will explore this example.

Next we turn to the proof of Theorem \ref{T1.2}.

\begin{proof} [Proof of Theorem \ref{T1.2}]  For $\kappa$ a measure of compact support on $\bbC$, we define, for all $z\in\bbC$, its antipotential by
\begin{equation}\label{2.7}
  \Phi_\kappa(z) = \int \log|z-w|\, d\kappa(w)
\end{equation}
(where the integral either converges or diverges to $-\infty$).  It is subharmonic and locally $L^1$ and behaves like $\left(\int d\kappa\right) \log|z|$ near infinity, so it defines a tempered distribution and its distributional Laplacian obeys
\begin{equation}\label{2.8}
   \Delta\Phi_\kappa = 2\pi\kappa
\end{equation}
(see \cite[Section 6.9]{RA} and \cite[Section 3.2]{HA}).

Now let $\mu_\infty$ be a limit point of the zero counting measure.  By \eqref{2.2}, $\Psi_{\mu_\infty}(z) = \Psi_{\mu_\fre}(z)$ for $z\in N\setminus\fre$.  Since the functions are $L^1$ and $N\cap\fre$ has Lebesgue measure zero, we conclude they define the same distributions on $N$. By \eqref{2.8}, $\mu_{\infty}\restriction N = \mu_{\fre}\restriction N$.  Since the restrictions of all limit points agree, we see the restrictions of the zero counting measures to $N$ converge and converge to  $\mu_{\fre}\restriction N$.
\end{proof}

For the case where one has a global assumption on $\fre$ (i.e.\ where $N$ is a very large disk), our result is somewhat weaker than that of Blatt--Saff--Simikani \cite{BSS1988} in that they only require that $\fre^{\inte}$ is empty while we require that $\fre$ have two-dimensional Lebesgue measure zero.  Their arguments are global and do no appear to work with only a local assumption.  On the other hand, Totik \cite{TotikPC} has sent us an example (reproduced below) of two distinct measures, $\mu$ and $\nu$, with $\Phi_\mu=\Phi_\nu$ off a set, $\fre$, with $\fre^{\inte}$ empty, so our method doesn't seem capable of extending to the case where one only supposes that $N\cap\fre^{\inte}$ is empty.

\begin{example}[\cite{TotikPC}]
Let $D_0=\bbD$ be the open unit disk and define recursively $D_{n}$ to be $D_{n-1}$ with a small closed disk removed. Assume that $0\in D_n$ and let $\mu_n$ be the balayage of $\de_0$ onto $\partial D_n$ so that the potentials of $\de_0$ and $\mu_n$ coincide on $\bbC\setminus\ol{D_n}$ (see \cite[Section II.4]{SaT} for the notion of balayage). The center of the removed disk can be chosen arbitrarily in $D_{n-1}\setminus\{0\}$ while the radius we choose small enough to ensure $0\in D_n$ and $\mu_n(\partial\bbD)>\f12$. Now choose centers of the removed disks in such a way that $\fre=\cap_n\ol{D_n}$ is nowhere dense and let $\mu$ be a weak limit of the $\mu_n$'s. Then both $\mu$ and $\de_0$ are supported on $\fre$, the potentials of $\mu$ and $\de_0$ are the same on $\bbC\setminus\fre$, and $\mu$ is not $\de_0$ since $\mu(\partial\bbD)\ge\f12$.
\end{example}

%%%%%%%%%%%%%%%%%%%%%%%%%%%%%%%%%%%%%%%%%%%%%%%%%%%%%%%%%%%%%%%
\section{Lemniscates} \lb{s3}
%%%%%%%%%%%%%%%%%%%%%%%%%%%%%%%%%%%%%%%%%%%%%%%%%%%%%%%%%%%%%%

We now turn to the study of when Widom factors, $\calW_n(\fre)=t_n/C(\fre)^n$, are bounded as $n\to\infty$ and explicit bounds on $\sup_n \calW_n(\fre)$.  In this section, we will prove Theorem \ref{T1.3}.  It is a very small addendum to our discussion of lemniscates in \cite{CSZ3}.  Solid lemniscates are defined by \eqref{1.7} where, without loss, we can suppose that $P_n$ is a monic polynomial of degree $n$. The Green's function, $G_{\frf_n}$, of $\frf_n$ is clearly given by
\begin{equation}\label{3.1}
  G_{\frf_n}(z) = \frac{1}{n} \log\left(\frac{|P_n(z)|}{\alpha}\right)
\end{equation}
from which it follows that
\begin{equation}\label{3.2}
  C(\frf_n)=\alpha^{1/n}
\end{equation}
Thus
\begin{equation}\label{3.3}
  \norm{P_n^k}_{\frf_n} = \alpha^k = C(\frf_{nk})^{nk} \le t_{nk}
\end{equation}
by the Szeg\H{o} lower bound.  Since $P_n^k$ is monic we see that
\begin{equation}\label{3.4}
  T_{nk} = P_n^k; \qquad \calW_{nk}(\frf_n) = 1
\end{equation}

\begin{proof} [Proof of Theorem \ref{T1.3}]  Since, for any compact set, $\fre$, $T_jT_\ell$ is a monic polynomial of degree $j+\ell$ with $\norm{T_jT_\ell}_\fre \le \norm{T_j}_\fre \norm{T_\ell}_\fre$, we see that $t_{j+\ell} \le t_j t_\ell$.  It follows that
\begin{equation*}
  \calW_{nk+j}(\frf_n) \le \calW_{nk}(\frf_n) \calW_j(\frf_n) = \calW_j(\frf_n)
\end{equation*}
proving that
\begin{equation}\label{3.5}
  \sup_m \calW_m(\frf_n) = \max_{j=0,\dots,n-1} \calW_j(\frf_n)
\end{equation}
which is the assertion of \eqref{1.9}
\end{proof}

%%%%%%%%%%%%%%%%%%%%%%%%%%%%%%%%%%%%%%%%%%%%%%%%%%%%%%%%%%%%%%
\section{Level Sets of Green's Functions} \lb{s4}
%%%%%%%%%%%%%%%%%%%%%%%%%%%%%%%%%%%%%%%%%%%%%%%%%%%%%%%%%%%%%%

In this section, we will prove Theorems \ref{T1.4}--\ref{T1.6}.  We start with Theorem \ref{T1.5}.

\begin{proof} [Proof of Theorem \ref{T1.5}] By definition of $\fre^\alpha$, we have that $G_{\fre^\al}=G_\fre-\al$ is the Green's function for $\fre^\al$, which implies that $C(\fre^\al)=e^\al C(\fre)$ and also $(\fre^\al)^\be=\fre^{\al+\be}$. %Since $(\fre^\al)^\be=\fre^{\al+\be}$
Thus, it suffices to show that $\calW_n(\fre^\al)\le \calW_n(\fre)$ for $\al>0$.

Let $T_n$ be the $n$th Chebyshev polynomial of $\fre$. Define 
\begin{align}
    \frf_n=\{z\in\bbC \st |T_n(z)|\le\|T_n\|_\fre\}
\end{align} 
Then, as discussed in the previous section, $T_n$ is also the $n$th Chebyshev polynomial for $\frf_n$ and $G_{\frf_n}(z)=\tfrac{1}{n}\log\tfrac{|T_n(z)|}{\|T_n\|_\fre}$ is the Green's function for $\frf_n$ so
\begin{equation}\label{4.1}
   |T_n(z)|=\|T_n\|_\fre \exp\{nG_{\frf_n}(z)\}
\end{equation}

Let $T^\al_n$ denote the Chebyshev polynomials for $\fre^\al$ and define $\frf_n^\al=\{z \st G_{\frf_n}(z)=\al\}$.  Since $\fre\subset\frf_n$, we have $G_\fre\ge G_{\frf_n}$ and hence $\fre^\al$ lies inside $\frf_n^\al$. It follows that $\|T^\al_n\|_{\fre^\al}\le \|T_n\|_{\fre^\al}\le \|T_n\|_{\frf_n^\al}$ and so, by \eqref{4.1}, $\|T^\al_n\|_{\fre^\al}\le \|T_n\|_\fre e^{n\al}$. Dividing by $C(\fre^\al)^n=e^{n\al}C(\fre)^n$ yields $\calW_n(\fre^\al)\le \calW_n(\fre)$.
\end{proof}

In the case of $\fre\subset\bbR$, the bound $\calW_n(\fre^\al)\le \calW_n(\fre)$, $\al\ge0$, can be improved:

\begin{proof} [Proof of Theorem \ref{T1.6}]
Let $\fre_n=\{x\in\bbR \st |T_n(x)|\le\|T_n\|_\fre\}$. Then, by (2.4) in \cite{CSZ1}, we have that
\begin{equation}\label{4.2}
   |T_n(z)|\le\tfrac{1}{2}\|T_n\|_\fre\bigl(\exp\{nG_{\fre_n}(z)\}+ \exp\{-nG_{\fre_n}(z)\}\bigr)
\end{equation}

Let $T_n^\al$ denote the Chebyshev polynomial of $\fre^\al$. Since $\fre\subset\fre_n$, we have $G_\fre\ge G_{\fre_n}$ and hence $\fre^\al$ lies inside $\fre_n^\al$. It follows that $\|T_n^\al\|_{\fre^\al}\le \|T_n\|_{\fre^\al}\le \|T_n\|_{\fre_n^\al}\le \tfrac{1}{2}\|T_n\|_\fre(e^{n\al}+e^{-n\al})$ using \eqref{4.2} on $\partial\fre^\al$.

Thus, by \eqref{1.6}, we have that
\begin{align}
  t_n^{(\fre^\al)} &\le \exp\{ PW(\fre) \} (e^{n\al}+e^{-n\al}) C(\fre)^n \nonumber \\
                   &= \exp\{ PW(\fre) \} (1+e^{-2n\al}) C(\fre^\al)^n \label{4.3}
\end{align}
\end{proof}

To prove Theorem \ref{T1.4}, we need a complex variant of the Alternation Theorem (\cite[Theorem 1.1]{CSZ1}):

\begin{lemma} \lb{L4.1} Suppose $K\subset\bbC$ is a compact set and $P$ is a monic degree $n$ polynomial such that $K_0=\{z\in K \st P(z)=\pm\|P\|_K\}$ contains $2n$ points counting multiplicities. Then $P$ is the $n$th Chebyshev polynomial of $K$.
\end{lemma}

\begin{proof}
Note that $T_1(z)=z$ is the Chebyshev polynomial of the two point set $\fre=\{\pm\|P\|_K\}$. Thus, by \cite[Theorem 6.1]{CSZ3}, $P=T_1\circ P$ is the Chebyshev polynomial of $\fre_P=P^{-1}(\fre)=\{z\in\bbC \st P(z)\in\fre\}\supset K_0$. Since $P$ has degree $n$, $\fre_P$ consists of $2n$ points. Hence $\fre_P=K_0\subset K$. Since $\|P\|_K=\|P\|_{\fre_P}$, it follows that $P$ is also the $n$th Chebyshev polynomial of $K$.
\end{proof}

\begin{proof} [Proof of Theorem \ref{T1.4}]
Since a period-$n$ set is also a period-$kn$ set for each $k=1,2,\dots$, it suffices to prove the result for $k=1$. As in \cite[(2.4)]{CSZ1}, we have that
\begin{equation}\label{4.4}
    T_n(z)=\tfrac{1}{2}\|T_n\|_\fre\bigl(B_n(z)^n+B_n(z)^{-n}\bigr)
\end{equation}
where $B_n$ is the analytic multi-valued Blaschke function defined as a complexification of $|B_n|=\exp\{-G_{\fre_n}\}$. Then $|B_n^{-n}|=e^{n\al}$ on $\fre_n^\al$ and hence the extremal values of $T_n$ on $\fre_n^\al$ are $\pm\cosh(n\al)\|T_n\|_\fre$ which occur at the $2n$ points $\{z\in\fre_n^\al \st B_n(z)^{-n}=\pm e^{n\al}\}$. Thus, by the lemma, $T_n$ is the $n$th Chebyshev polynomial of $\fre_n^\al$.
\end{proof}

%%%%%%%%%%%%%%%%%%%%%%%%%%%%%%%%%%%%%%%%%%%%%%%%%%%%%%%%%%%%%%
\section{Do Totik--Widom Bounds hold for the Connected, Simply Connected Case} \lb{s5}
%%%%%%%%%%%%%%%%%%%%%%%%%%%%%%%%%%%%%%%%%%%%%%%%%%%%%%%%%%%%%%

From the time we proved that all Parreau--Widom sets (henceforth PW) in $\bbR$ have the TW property, whether this result extends to $\fre\subset\bbC$ has been an interesting open question.  Initially, we thought it was likely true.  We realized that a key test case was where $\fre$ is a connected, simply connected (henceforth CSC) set.  In that case, it is a consequence of the Riemann mapping theorem that $G_\fre$ has no critical points on $\bbC\setminus\fre$, so the PW condition holds.  If PW$\Rightarrow$TW for general $\fre\subset\bbC$, then clearly every CSC set obeys TW.  And if PW$\Rightarrow$TW is false, it likely fails for some CSC set.

So, for several years, we have discussed widely the need to look at this question for CSC sets.  It goes back to Faber \cite{Faber} that if $\fre$ is Jordan region with analytic boundary, then $\lim_{n\to\infty} \calW_n(\fre)=1$ so TW holds.  Widom \cite{Widom} extended this to $C^{2+\veps}$ boundary.

We suggested in several talks that if TW fails, it likely fails for the Koch snowflake but this set is more regular that one might think -- it is a quasidisk.  Andrievski \cite{And} and Andrievski--Nazarov \cite{AndNaz} proved that every quasidisk has the TW property, so the Koch snowflake does not provide a counterexample.  

Here, we want to suggest several additional places to look for counterexamples.

(1) \emph{Koch antennae.} Recall the construction of the Koch snowflake.  One starts with $T_1$, a solid equilateral triangle in $\bbC$ with side $1$.  One adds, $T_2$, the three equilateral triangles of side $\tfrac{1}{3}$ centered on the midpoints of the sides of $T_1$.  Then $K_2=T_1\cup T_2$ has $4\times 3$ sides with size $\tfrac{1}{3}$.  At stage $j$, $K_j=K_{j-1}\cup T_j$ has $3\times 4^{j-1}$ sides of size $3^{-(j-1)}$.  $T_{j+1}$ is then the $3\times 4^{j-1}$ triangles with side $3^{-j}$ centered at the midpoints of the sides of $K_j$.  $K_\infty=\cup_{j=1}^\infty K_j$ is the Koch snowflake, a Jordan region whose boundary is a non-rectifiable curve of Hausdorff dimension strictly greater than $1$.  But it is regular in the sense that it is a quasidisk.

Modify this construction by picking $a_1, a_2,\dots$ all in $(0,1]$.  $K_j$ still has $3 \times 4^{j-1}$ sides but now of size $s_j$ defined inductively starting with $s_1=1$.  The $3\times 4^{j-1}$ triangles of $T_{j+1}$ are now isosceles with base $\tfrac{1}{3}a_js_j$ and two equal sides $\tfrac{1}{2}(1-\tfrac{1}{3}a_j)s_j=s_{j+1}$.  The limit $K_\infty$ is still a Jordan region with non-rectifiable boundary of dimension larger than $1$.  With the case $a_j\downarrow 0$ rapidly in mind, we call this the Koch antenna (although, so far as we know, Koch never considered it!).  If $\liminf a_j = 0$, $K_\infty$ is not a quasidisk and \cite{And,AndNaz} do not apply.  We believe that the case $a_j=3^{-j}$ is a good candidate for a situation where TW might fail.  An extreme case is what happens if all $a_j=0$ so the added ``triangles'' are line segments (we need to destroy the symmetry by taking $s_{j+1}=\beta s_j$ with $\beta$ strictly less than $1/2$ to avoid the lines in $T_j$ from intersecting).  The boundary is no longer a Jordan curve although it is the image of a circle under a continuous map.

(2) \emph{The Cauliflower.}  The Cauliflower is the Julia set of the map $z^2+z$; see, for example, Milnor \cite[Figure 2.4]{Milnor}.  This has inward pointing cusps so, by \cite{SS2015}, the density of zeros approaches the equilibrium measure. Since there has been previous work \cite{BGH1, BGH2, Bessis, KB, Alpan, AG} on extremal polynomials on Julia sets (albeit certain disconnected Julia sets where PW fails), this might be an approachable example.

(3) \emph{Non-Jordan Regions.} All examples discussed so far in the context of TW bounds have been Jordan regions in that $\partial\fre$ is a simple closed continuous curve.  Examples like the lemniscate, the extreme antenna (i.e.\ all $a_j=0$) or even a disk with a spike ($\overline{\bbD}\cup [1,2]$) aren't Jordan regions but at least their boundaries are images of a continuous curve.  But there are CSC regions whose boundaries are not images of continuous curves or even boundaries with inaccessible points.  A good example is the open set
\begin{equation}\label{5.1}
  \Omega=(0,1)\times(0,1)\setminus \bigcup_{n=1}^\infty\bigl\{(1-\tfrac{1}{2n},y) \cup(1-\tfrac{1}{2n+1},1-y)\,|\,y\in[0,\tfrac{3}{4}]\bigr\}
\end{equation}
of \cite[Figure 8.2.1]{BCA}.  Of course, $\Omega$ is open and $\overline{\Omega}$ is a Jordan region but $\fre = \{(z-z_0)^{-1} \st z \in\bbC\setminus\Omega\}$, where $z_0\in\Omega$, is a compact set whose boundary has tangled spikes and the boundary is not continuous nor everywhere accessible from the outside.  Our point here is not that this example should be analyzed but that while searching for possible counterexamples to ``every CSC set is TW'', one needs to consider sets whose boundary is not a continuous curve.

In any event, we regard finding either a non-TW example among the CSC sets or else proving that all CSC sets are TW one of the most important open questions in the theory of Chebyshev polynomials.
%%%%%%%%%%%%%%%%%%%%%%%%%%%%%%%

\end{document}